\documentclass{amsart}
\usepackage{amssymb}
\usepackage{amsmath,amscd}
\usepackage{graphicx}
\usepackage{color,hyperref}
\usepackage{tikz}
\usetikzlibrary{matrix,arrows}
\usepackage{enumitem}

\newcommand{\mf}{\mathfrak}

\newcommand{\mb}{\mathbf}

\newcommand{\R}{\mathbf R}
\newcommand{\C}{\mathbf C}

\newcommand{\Z}{\mathbf Z}
\newcommand{\F}{\mathbf F}
\newcommand{\N}{\mathbf N}

\newcommand{\Adj}{\textnormal{Adj}}

\newcommand{\Jac}{\textnormal{Jac}}

\newcommand{\Stab}{\textnormal{Stab}}

\numberwithin{equation}{section}

\theoremstyle{plain}
\newtheorem{theorem}{Theorem}
\newtheorem{lemma}[theorem]{Lemma}

\theoremstyle{definition}
\newtheorem{definition}[theorem]{Definition}
\newtheorem{example}[theorem]{Example}

\begin{document}

\title[Jacobian Matrices]{On the Jacobian Matrices of Generalized Chebyshev Polynomials}

\author[\.{I}LER\.{I}]{AHMET \.{I}LER\.{I}}
\address{Middle East Technical University, Mathematics Department, 06800 Ankara,
	Turkey.}
\curraddr{}
\email{komer@metu.edu.tr}
%\thanks{thanks from first author}

%    author two information
\author[K\"{U}\c{C}\"{U}KSAKALLI]{\"{O}MER K\"{U}\c{C}\"{U}KSAKALLI}
\address{Middle East Technical University, Mathematics Department, 06800 Ankara,
	Turkey.}
\curraddr{}
\email{ahmet.ileri@metu.edu.tr}
%\thanks{special thanks}

\date{\today}

\begin{abstract}
In this paper, we give a practical method to compute the Jacobian matrices of generalized Chebyshev polynomials associated to arbitrary semisimple Lie algebras. The entries of each Jacobian matrix can be expressed as a linear combination of characters of irreducible representations of the underlying Lie algebra with integer coefficients. These integer coefficients can be obtained by basic computations in the fundamental Weyl chamber.
\end{abstract}

\dedicatory{In memory of James E. Humphreys. (1939-2020)}



\maketitle

\section{Introduction}
The generalized Chebyshev polynomials are polynomial mappings $P_{\mf{g}}^k: \C^n \rightarrow \C^n$ with integer coefficients obtained from the exponential invariants of arbitrary semisimple Lie algebras $\mf{g}$ of rank $n$. From their definition, they naturally commute
\[P_{\mf{g}}^k \circ P_{\mf{g}}^l = P_{\mf{g}}^l \circ P_{\mf{g}}^k\]
and it is believed that they exhaust all commuting polynomials under certain additional assumptions \cite{veselov-survey}. They are orthogonal with respect to a certain measure and can be extended to a complete set of orthogonal polynomials \cite{hoffwith}.

The present manuscript has its origin in an attempt to classify (arithmetically) exceptional polynomial mappings with two or more variables. We recall that a polynomial mapping $P \in {\Z}[\mb{x}]$ in $n$ variables is said to be (arithmetically) exceptional if the reduced map $\bar{P} : {\F}_p^n \to {\F}_p^n$ is a permutation for infinitely many primes $p$. The classification of exceptional polynomials with one variable is finished \cite{fried}. They are the compositions of linear polynomials, power maps and Chebyshev polynomials. The ideas of Fried can be extended to the projective setting by translating exceptionality to a property of permutation groups \cite{guralnick}.  

The elementary symmetric polynomials and the power-sum symmetric polynomials both  generate the algebra of symmetric polynomials. Using this basic idea, Lidl and Wells proved the existence of polynomial mappings of arbitrary rank which are exceptional \cite{lidlwells}. This basic construction of Lidl and Wells can be related to the simple complex Lie algebras $A_n$ \cite{hoffwith}. In a previous work of the second author, it is proved that the generalized Chebyshev polynomials $P_{\mf{g}}^k$ are exceptional for any prime $k>e+1$ where $e$ is the exponent of the Weyl group \cite{kucuk-bulletin}. 

We hope to use the theory of Lie algebras to understand the classification problem of exceptional polynomials by possibly eliminating the need for group theoretical techniques not available for higher ranks. In that case, we believe that the Jacobian matrices of generalized Chebyshev polynomials would be a key tool since they determine the ramification locus.

The organization of the paper is as follows. In the second section, we give some basic notation and terminology about the root systems together with some basic results that will be used in further sections. In the third section, we review the theory of exponential invariants and provide a proof of a theorem of R.~Steinberg that we believe to have remained unpublished.  In the fourth section, we give the definition of generalized Chebyshev polynomials. In the fifth section, we state and prove our main result, and provide some examples of low rank. 

\section{Notation and Terminology}
In this section, we give some basic notation and terminology. We will also state some results that are essential in the rest of the manuscript. The main references are \cite{hump-lie} and \cite{hump-ref}. Let $E$ be an $n$-dimensional (real) Euclidean vector space endowed with a positive definite symmetric bilinear form. For any nonzero vector $\alpha\in E$, let $H_\alpha$ be the hyperplane through the origin orthogonal to the line $L_\alpha=\R\alpha$. The reflection in the hyperplane $H_\alpha$ is given by
\[\sigma_\alpha (\beta) = \beta - \frac{2(\beta,\alpha)}{(\alpha,\alpha)}\alpha\]  
The number $2(\beta,\alpha)/(\alpha,\alpha)$ appears frequently and it is abbreviated by $\langle \beta, \alpha \rangle$. A subset $\Phi$ of $E$ is called a root system in $E$ if the following axioms are satisfied:
\begin{enumerate}[label=(R\arabic*)]
	\item $\Phi$ is finite, spans $E$, and does not contain $0$.
	\item If $\alpha \in \Phi$, the only multiples of $\alpha$ in $\Phi$ are $\pm \alpha$.	
	\item If $\alpha\in \Phi$, then the reflection $\sigma_\alpha$ leaves $\Phi$ invariant.
	\item If $\alpha,\beta \in \Phi$, then $\langle \beta, \alpha \rangle \in \Z$.
\end{enumerate}

The elements of $\Phi$ are called roots because of their historical connection to the semisimple Lie algebras. Let $W$ be the subgroup of GL$(E)$ generated by the reflections $\sigma_\alpha, \alpha\in\Phi$. This subgroup $W$ is called the Weyl group of the root system $\Phi$ and it is an example of a finite reflection group. 

There are other examples of finite reflection groups which do not occur as Weyl groups. The remaining cases become available by removing the axiom (R4) and allowing non-crystallographic reflection groups. It is not essential to distinguish roots as longer or shorter to define a finite reflection group. On the other hand, the Weyl group of different semisimple Lie algebras $B_n$ and $C_n$ turn out to be isomorphic. In this manuscript, we will be focusing on the following finite reflection groups: $A_n, B_n / C_n, D_n, G_2, F_4, E_6, E_7, E_8$. 

The algebra of polynomial functions on $E$ is the symmetric algebra $S(E^*)$ of the dual space $E^*$. The symmetric algebra $S(E^*)$ may be identified with the polynomial ring $\R[\mb{x}]=\R[x_1,\ldots,x_n]$ where the $x_i$ are the coordinate functions. A finite reflection group $W \subset \textnormal{GL}(E)$ acts naturally on $\R[\mb{x}]$ by the rule 
$$(wf)(\gamma)=f(w^{-1}\gamma)$$ 
where $w\in W, \gamma\in E$. We say that a polynomial $f\in \R[\mb{x}]$ is $W$-invariant if $wf=f$ for all $w\in W$. 

\begin{theorem}[Chevalley's Theorem \cite{chevalley}]\label{chevalley} The subalgebra $\R[\mb{x}]^W$ of $W$-invariants is generated as an $\R$-algebra by $n$ homogeneous, algebraically independent elements of positive degree (together with $1$). 
\end{theorem}
 
Even though a set of generators $\{f_1,\ldots, f_n\}$ for $\R[\mb{x}]^W$ is not unique, the degrees $\{d_1,\ldots, d_n\}$ are independent of the choice of generators. It is well known that the size of the Weyl group is obtained by the product of degrees $d_i$. See Table~\ref{deg-pol-inv}. 
\begin{table}[!h]
\[\begin{array}{|c|c|c|} \hline
	\textnormal{Type} & d_1, \ldots, d_n & |W|=\Pi d_i\\ \hline
	A_n & 2,3,4,\ldots, n+1 & (n+1)!\\ 
	B_n/C_n & 2,4,6\ldots, 2n & 2^nn!\\ 
	D_n & 2,4,6\ldots, 2n-2,n & 2^{(n-1)}n!  \\ 
	E_6 & 2,5,6,8,9,12 & 2^73^45  \\ 
	E_7 & 2,6,8,10,12,14,18 & 2^{10}3^55\,7 \\ 
	E_8 & 2,8,12,14,18,20,24,30 & 2^{14} 3^55^27 \\ 	
	F_4 & 2,6,8,12 & 2^73^2 \\ 
	G_2 & 2,6 & 12\\ \hline
\end{array}\]
\caption{\label{deg-pol-inv}The degrees of polynomial invariants.}
\end{table}

Another important quantity that is independent of the choice of generators is the Jacobian determinant. Recall that the Jacobian matrix is defined by
\[\frac{\partial (f_1,\ldots,f_n)}{\partial (x_1,\ldots,x_n)}= \begin{bmatrix}
	\dfrac{\partial f_1}{\partial x_1} & \cdots & \dfrac{\partial f_1}{\partial x_n}\\
	\vdots                             & \ddots & \vdots\\
	\dfrac{\partial f_n}{\partial x_1} & \cdots & \dfrac{\partial f_n}{\partial x_n}
\end{bmatrix} = \left[\frac{\partial f_i}{\partial x_j}\right].\]

A subset $\Delta$ of a root system $\Phi$ is called a base if $\Delta$ is a vector space basis of $E$ and each root can be written as $\sum k_\alpha \alpha, \alpha\in \Delta$ with integral coefficients $k_\alpha$ all nonnegative or all nonpositive. The roots in $\Delta$ are called simple. A base always exists and the root system can be partitioned into two subsets, namely the positive roots and the negative roots. The positive are denoted by $\Phi^+$.

\begin{theorem}\label{detjac} 
Fix a set of generators $\{f_1, \ldots, f_n\}$ for the algebra $\R[\mb{x}]^W$. For each $\alpha\in \Phi$, let $l_\alpha$ be a linear polynomial whose zero set is the hyperplane $H_\alpha$. Then 
\[ \det\left(\left[\frac{\partial f_i}{\partial x_j}\right]\right)= c\prod_{\alpha \in \Phi^+} l_\alpha \]
for some constant $c\in\R$, depending on the choices of $f_i$ and $l_\alpha$.
\end{theorem}
 
Let $\mf{g}$ be a complex finite semisimple Lie algebra. It is well known that $\mf{g}$ is a direct sum of simple Lie algebras: $A_n, B_n, C_n, D_n, E_6, E_7, E_8, F_4, G_2$.  The root space decomposition of $\mf{g}$ comes with a natural root system $\Phi$ attached to $\mf{g}$.  The axiom (R4) can be stated in a simpler fashion by introducing coroots. For any root $\alpha$ in $\Phi$, the associated coroot is defined by
\[\alpha^\vee = \frac{2\alpha}{(\alpha , \alpha)}.\]
The axiom (R4) is equivalent to $(\beta,\alpha^\vee)\in \Z$ for all $\alpha,\beta \in \Phi$. 

Fix an ordering $(\alpha_1, \ldots, \alpha_n)$ of simple roots. The matrix 
\[[ \langle \alpha_i , \alpha_j \rangle]=[ (\alpha_i , \alpha_j^\vee )]\]
is called the Cartan matrix of $\Phi$. Its entries are called the Cartan integers. The Dynkin diagram is an alternative object which includes the same information as the Cartan matrix. We will be using the Cartan matrices and Dynkin diagrams of \cite{hump-lie}. For example, the rank two simple Lie algebras are listed as follows:
\[\begin{array}{ccc} 
 A_2 & B_2 & G_2 \\ 
 \includegraphics[scale=0.5]{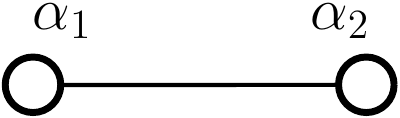}
& \includegraphics[scale=0.5]{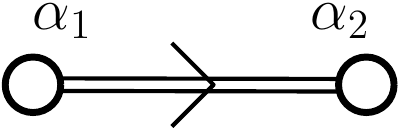}& \includegraphics[scale=0.5]{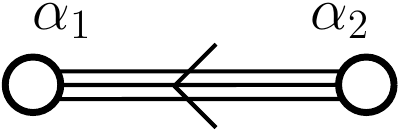}\\ 
\quad \left[\begin{array}{cc}2&-1\\-1&2\end{array}\right] \quad&
\quad \left[\begin{array}{cc}2&-2\\-1&2\end{array}\right] \quad&
\quad \left[\begin{array}{cc}2&-1\\-3&2\end{array}\right] \quad \\ 
\end{array}\]

The representation theory of Lie algebras has a central theme, namely the highest weight. We will be using the Weyl character formula, 
Theorem~\ref{Weyl-Char-Form}, as a main tool. Thus the fundamental weights $\omega_i$ are essential for us. They are defined by the following equation:
\[(\omega_i , \alpha_j^\vee) = \delta_{ij}\]
Here $\delta_{ij}$ is the Kroneckter delta function and $1 \leq i,j \leq n$. Note that the Cartan matrix transforms the fundamental weights into the simple roots. The following basic fact will be used several times to manage some important matrix multiplications. 

\begin{lemma}\label{inner} 
	The inner product $(\lambda,\gamma)$ can be computed by the identity
	\[(\lambda,\gamma) = \sum_{m=1}^{n}(\lambda,\alpha_m^\vee)(\omega_m, \gamma).\]
\end{lemma}
\begin{proof}
We write $\lambda=\sum_{i=1}^{n} a_i\omega_i$ and $\gamma=\sum_{j=1}^{n} b_j\alpha_j^\vee$. Using $(\omega_i , \alpha_j^\vee) = \delta_{ij}$, we see that $a_m=(\lambda, \alpha_m^\vee)$ and $b_m=(\omega_m,\gamma)$. Therefore $(\lambda,\gamma) = \sum_{m=1}^{n}a_mb_m$. 
\end{proof}

The hyperplanes $H_\alpha$ partition $E$ into finitely many regions. One of them has a special name. The fundamental Weyl chamber, relative to $\Delta$, denoted $\mathfrak{C}$, is the open convex set consisting of all $\gamma\in E$ which satisfy the inequalities $(\gamma,\alpha) > 0, \alpha\in\Delta$. 

Let $s_i$ be the size of the orbit $W(\omega_i)$. Since $(\omega_i , \alpha_j^\vee) = \delta_{ij}$, the stabilizer group $\Stab(\omega_i)$ is given by $W_i = \langle \sigma_{\alpha_j} \mathrel{|} \alpha_j\in \Delta\setminus\{\alpha_i\}\rangle$. It turns out that $W_i$ is the Weyl group of the root system with base $\Delta\setminus\{\alpha_i\}$. The orbit-stabilizer formula implies that $s_i = |W|/|W_i|$.

The length of $w$ (relative to $\Delta$) is the smallest integer $r$ for which $w$ can be expressed as $w=\sigma_1\cdots\sigma_r$ with simple reflections $\sigma_{\alpha_i}$ with $\alpha_i\in \Delta$. The number of positive roots sent to the negative roots by $w\in W$ is equal to the length of $w$.

There is a special weight $\rho=\omega_1+\ldots+\omega_n$ that appears frequently. It turns out that $\rho$ is equal to the half of the sum of positive roots. We observe that $\Stab(\rho)=\{1\}$ since it can contain only elements of length zero. Similarly $\Stab(\rho-\omega_i)$ can contain only elements of length less than or equal to one. On the other hand $\sigma_{\alpha_i} (\rho-\omega_i) = \rho-\omega_i$. Thus we conclude that $\Stab(\rho-\omega_i) = \{1,\sigma_{\alpha_i}\}$. 

The following fact enables us to realize the Weyl group as a subgroup of matrices with integer entries. 

\begin{lemma}\label{glz}
	Set $T_w=[(\omega_i,w(\alpha_j^\vee))]$ for each $w\in W$. Then the map $w \mapsto T_w$ is an injective group homomorphism from $W$ into \textnormal{GL}$(n,\mathbf{Z})$. Moreover $\det(w)=\det(T_w)$.
\end{lemma}
\begin{proof}
	The map $w\mapsto T_w$ is a group homomorphism because 
	\begin{align*}
		T_{w}T_{w'} &= [(\omega_i,w(\alpha_j^\vee))][(\omega_i,w'(\alpha_j^\vee))] \\
		& = [(w^{-1}(\omega_i),\alpha_j^\vee)][(\omega_i,w'(\alpha_j^\vee))]\\
		& = [(w^{-1}(\omega_i),w'(\alpha_j^\vee)]\\
		& = [(\omega_i,ww'(\alpha_j^\vee)]\\
		& = T_{ww'}.
	\end{align*}
Here, the second and the fourth equalities hold since $w$ is an isometry. The third equality is obtained by applying Lemma~\ref{inner}. If $T_w=[(\omega_i,w(\alpha_j^\vee))]=[\delta_{ij}]$ then $w(\alpha_j^\vee) = \alpha_j^\vee$ for each $j$. Since coroots span $E$, we must have $w=1$. The Cartan matrix transforms the fundamental weights into the simple roots, and $\sigma_{\alpha_i}(\omega_j) = \omega_j - \delta_{ij}\alpha_i$ for each $\alpha_i \in \Delta$. Thus, the matrix $T_{w}$ for $w=\sigma_{\alpha_i}$ is obtained by subtracting the $i$th row of the Cartan matrix from the identity matrix. Such a matrix has integer entries and has determinant minus one. The Weyl group is generated by $\sigma_{\alpha_i}$, and therefore $T_w$ has integer entries for each $w\in W$.
\end{proof}

\begin{example} Suppose that the root system has type $G_2$ with the Cartan matrix above. In this case, the inverse of Cartan matrix has integer entries. We have	
\[ \begin{array}{rcl}	
		\alpha_1 &=& 2\omega_1 - \omega_2,\\
		\alpha_2 &=& -3\omega_1 +2\omega_2,\end{array}
	\quad \text{ and } \quad
	\begin{array}{rcl}	
		\omega_1 &=& 2\alpha_1 + \alpha_2,\\
		\omega_2 &=& 3\alpha_1 + 2\alpha_2.\end{array} 	
	\]
The Weyl group $W$ is generated by the reflections $\sigma_{\alpha_1}$ and $\sigma_{\alpha_2}$ and it has 12 elements. It can be realized as a subgroup of GL$(2,\Z)$ with the following generators: 
$$T_{\sigma_{\alpha_1}} = \left[\begin{array}{cc}-1&1\\0&1\end{array}\right], \textnormal{ and } \quad T_{\sigma_{\alpha_2}} = \left[\begin{array}{cc}1&0\\3&-1\end{array}\right].$$
The orbits of $\omega_1$ and $\omega_2$, both with 6 elements, are
\[W(\omega_1)=\{\omega_1,-\omega_1,\omega_1-\omega_2,-\omega_1+\omega_2,2\omega_1-\omega_2,-2\omega_1+\omega_2\},\]
\[W(\omega_2)=\{\omega_2,-\omega_2,3\omega_1-\omega_2,-3\omega_2+\omega_2,3\omega_1-2\omega_2,-3\omega_1+2\omega_2\}.\]
Note that $\Stab(\omega_1)=\{1,\sigma_2\}$ and $\Stab(\omega_2)=\{1,\sigma_1\}$.

The fundamental Weyl chamber $\mathfrak{C}$ relative to $\Delta=\{\alpha_1, \alpha_2\}$ is highlighted with gray color in Figure~\ref{fig:RootG2}. 
\begin{figure}[htbp]
	\centering
	\includegraphics[scale=0.8]{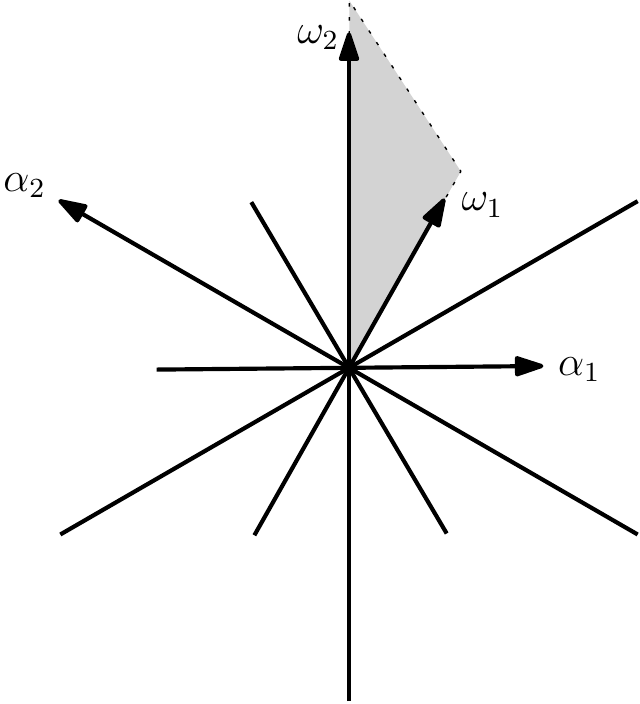}
	\caption{The root system for $G_2$.}
	\label{fig:RootG2}
\end{figure}	
\end{example}

\section{Exponential Invariants}
The main reference for this section is Bourbaki \cite[Ch. 6, \S3]{Bourbaki}. Let $\Lambda$ be the free abelian group generated by the fundamental weights. The group algebra of $\Lambda$ over a unique factorization domain $A$ is denoted by $A[\Lambda]$. It consists of formal sums
$$a=\sum_{\lambda\in \Lambda} a_\lambda e^\lambda$$ 
with coefficients $a_\lambda\in A$. The exponential notation is used to distinguish two different additive structures. We have
\[e^\lambda e^{\lambda'} = e^{\lambda+\lambda'},\quad (e^\lambda)^{-1}=e^{-\lambda},\quad e^0=1.\] 
The Weyl group acts on $\Lambda$ and therefore on the group algebra by $w( e^\lambda) = e^{w (\lambda)}$. 

A fixed ordering of simple roots $\alpha_i \in \Delta$ provides a partial order on $E$. If $\lambda, \lambda' \in \Lambda$, then $\lambda \geq \lambda'$ if and only if $\lambda-\lambda'$ is a linear combination of $\alpha_i$ with nonnegative coefficients. Let $a=\sum_{\lambda\in\Lambda}a_\lambda e^\lambda$ be an element of $A[\Lambda]$. The set $S$ of $\lambda \in \Lambda$ such that $a_\lambda \neq 0$ is called the support of $a$ and the set $M$ of maximal terms of $S$ is called the maximal support of $a$. A term $a_\lambda e^\lambda$ with $\lambda \in M$ is called a maximal term of $a$.  

\begin{lemma}\label{maxterm}
Let $a$ be an element of $A[\Lambda]$ with maximal terms $M_a=\{a_\lambda e^\lambda\mathrel{|} \lambda \in X\}$. If $b$ is an element of $A[\Lambda]$ with unique maximal term $e^\mu$, then the product $ab$ has maximal terms $\{a_\lambda e^{\lambda+\mu}\mathrel{|} \lambda \in X\}$
\end{lemma}

Recall that $W$ is a subgroup of GL$(E)$. We have $\det(w)=\pm 1$ for each $w$ in $W$, since $W$ is generated by reflections. 

\begin{definition}
	An element $a\in A[\Lambda]$ is said to be anti-invariant under $W$ if $w(a)=\det(w)a$ for all $w\in W$.
\end{definition}

The anti-invariant elements of $A[\Lambda]$ form a submodule. For any $a\in A[\Lambda]$, put 
\[J(a)=\sum_{w\in W} \det(w)w(a).\]
If $|W|$ is invertible in $A$, then $\frac{1}{|W|}J$ is a projection from $A[\Lambda]$ onto the submodule of anti-invariant elements. The fundamental Weyl chamber $\mathfrak{C}$, relative to $\Delta$, enables us to write a natural basis for the submodule of anti-invariant elements of $A[\Lambda]$. 

\begin{lemma}
If $\lambda\in\Lambda \cap \mf{C}$, then $w(\lambda) < \lambda$ for all $w\neq 1$ and $e^\lambda$ is the unique maximal term of $J(e^\lambda)$. Moreover, the elements $J(e^\lambda)$ form a basis of the submodule of anti-invariant elements of $A[\Lambda]$.
\end{lemma}

Recall that $\rho = \omega_1+\ldots+\omega_n$. The element $J(e^\rho)$ is a common divisor of anti-invariant elements. Conversely, the multiplication by $J(e^\rho)$ is a bijection from the submodule of invariant elements $A[\Lambda]^W$ to the submodule of anti-invariant elements. In particular, $J(e^{\rho+\lambda})/J(e^\rho)$ is a $W$-invariant element with unique maximal term $e^\lambda$. Alternatively, define
\[S(e^\lambda) = \sum_{\mu \in W(\lambda)}e^\mu.\]
The element $S(e^\lambda)$ is also a $W$-invariant element with unique maximal term $e^\lambda$. Both families form a basis for the submodule $A[\Lambda]^W$.
\begin{lemma}
If $\lambda\in\Lambda \cap \overline{\mf{C}}$, then $w(\lambda) \leq \lambda$ for all $w\in W$, and $e^\lambda$ is the unique maximal term of $S(e^\lambda)$ (or $J(e^{\rho+\lambda})/J(e^\rho)$). Moreover, the elements $S(e^\lambda)$ (or $J(e^{\rho+\lambda})/J(e^\rho)$) for $\lambda\in\Lambda \cap \overline{\mf{C}}$ form a basis of the submodule $A[\Lambda]^W$.
\end{lemma}
The finite sets $\{S(e^{\omega_i}) \mathrel{|} 1 \leq i \leq n\}$ or $\{J(e^{\rho+\omega_i})/J(e^\rho) \mathrel{|} 1 \leq i \leq n\}$ both generate $A[\Lambda]^W$ as an algebra. We refer to the following theorem as the exponential form of Theorem~\ref{chevalley}.
\begin{theorem}\cite[Ch.~6, \S3, Th.~1]{Bourbaki}\label{expformche}
Let $\omega_1, \ldots, \omega_n$ be the fundamental weights corresponding to the chamber $\mf{C}$, and, for $1\leq i \leq n$, let $x_i$ be an element of $A[\Lambda]^W$ with $e^{\omega_i}$ as its unique maximal term.  Let $$\varphi:A[X_1,\ldots,X_n] \rightarrow A[\Lambda]^W $$
be the homomorphism from the polynomial algebra $A[X_1,\ldots,X_n]$ to $A[\Lambda]^W$ that takes $X_i$ to $x_i$. Then, the map $\varphi$ is an isomorphism.
\end{theorem}

An exponential analogue of Theorem~\ref{detjac} also exists. In order to express this theorem, we define a linear map on $A[\Lambda]$ by the formula
\[D_j\left(\sum_{\lambda}a_\lambda e^\lambda\right) = \sum_{\lambda} (\lambda,\alpha_j^\vee)a_\lambda e^\lambda.\]
It can be directly verified that $D_j$ is a derivation of $A[\Lambda]$ for each $1\leq j \leq n$. On the other hand, we can consider a formal exponential sum as a complex valued function by putting
$$e^\lambda(\gamma) \mapsto e^{-2\pi i(\lambda, \gamma)}$$
as in \cite[Lemma~4.1]{hoffwith}. In this regard, the operator $D_j$ becomes a partial derivative with respect to a certain coordinate function.

The following theorem was communicated to Bourbaki by R.~Steinberg as it is stated as a footnote for \cite[Ch.~6, \S3, Ex.~1]{Bourbaki}. To our knowledge, it has remained unpublished.

\begin{theorem}\label{steinberg}
	Let $x_1, \ldots, x_n$ be a family of elements of $A[\Lambda]^W$ satisfying the condition of Theorem~\ref{expformche}. Then
$\det\left(\left[D_j(x_i)\right]\right)=J(e^\rho)$.	
\end{theorem}	
\begin{proof}  Let $x_i=\sum_{\lambda\in \Lambda} a^i_\lambda e^\lambda$. Then 
 $$w(D_j(x_i)) = \sum_{\lambda\in \Lambda} a^i_\lambda (\lambda, \alpha_j^\vee) e^{w(\lambda)}.$$ 
 Each $w\in W$ is an isometry. Replacing $(\lambda, \alpha_j^\vee)$ with $(w(\lambda), w(\alpha_j^\vee))$ and applying Lemma~\ref{inner}, we obtain 
\[w([D_j(x_i)]) = \left[\sum_{m=1}^{n} \left(\sum_{\lambda\in \Lambda} a^i_\lambda (w(\lambda),\alpha_m^\vee)e^{w(\lambda)} \right)(\omega_m,w(\alpha_j^\vee))\right].\]
We want to express the right hand side of the above equation as a product of two matrices. If 
\[ A = \left[\left(\sum_{\lambda\in \Lambda} a^i_\lambda (w(\lambda),\alpha_j^\vee)e^{w(\lambda)} \right)\right] \quad \text{and} \quad B=\left[(\omega_i,w(\alpha_j^\vee))\right] \]
then the entries of the product $AB=C$ are given by
$C_{ij} = \sum_{m=1}^n A_{im}B_{mj}$. We recall that $B=T_w$ as in Lemma~\ref{glz}. Thus, we have
$$w([D_j(x_i)]) = A B=[D_j(w(x_i))] T_w.$$
We have $w(x_i)=x_i$ by the hypothesis. The determinant function is multiplicative, and $\det(w)=\det(T_w)$ by Lemma~\ref{glz}. It follows that $\det(D_j(x_i))$ is anti-invariant.

Secondly, we show that $\det(D_j(x_i))$ has unique maximal term $e^\rho$. We first note that $D_j(x_i)$ has maximal terms less than or equal to $e^{\omega_i}$. Moreover, the derivation $D_j$ satisfies the following property
\[D_j(e^{\omega_i})=\delta_{ij}e^{w_i}.\]
It follows that $D_j(x_i)$ has unique maximal term $e^{\omega_i}$ if and only if $i=j$. Thus the diagonal summand $\prod D_j(x_j)$ of  $\det(D_j(x_i))$ has unique maximal term $e^\rho= e^{\omega_1+\ldots+\omega_n}$ by Lemma~\ref{maxterm}.

Each summand of the determinant is a product of $n$ terms and each derivation $D_j$ occurs once and only once. Moreover, each summand must contain a diagonal term $D_j(e^{\omega_j})$. Applying Lemma~\ref{maxterm}, we see that $\det(D_j(x_i))$ has maximal terms less than or equal to $e^\rho= e^{\omega_1+\ldots+\omega_n}$. This finishes the proof of the fact that $\det(D_j(x_i))$ has unique maximal term $e^\rho$.

The elements $J(e^\lambda)$, with $\lambda \in \Lambda \cap \mf{C}$, form a basis of the submodule of anti-invariant elements. Thus we must have $\det\left(\left[D_j(x_i) \right]\right) = J(e^\rho)$.
\end{proof}

A well known identity for $J(e^\rho)$, which is true in $A[\frac{1}{2} \Lambda]$, is the following
\[ J(e^\rho)=\prod_{\alpha \in \Phi^+} (e^{\alpha/2}-e^{-\alpha/2}). \]
Note that $e^\alpha(\gamma) = e^{-2\pi i(\alpha, \gamma)}=1$ if and only if $(\alpha,\gamma)=0$. It follows that the zero locus of $e^{\alpha/2} - e^{-\alpha/2}$, regarded as a complex valued function, is the hyperplane $H_\alpha$. This is also the case for Theorem~\ref{detjac}. The analogy between the polynomial and the exponential invariants is proved to be beneficial. For example, the exponents of a reflection group based on the height of roots in the crystallographic root systems can be computed in this fashion \cite[Chap.~10]{carter}.

Generating sets $\{S(e^{\omega_i}) \mathrel{|} 1\leq i \leq n\}$ and  $\{J(e^{\rho+\omega_i})/J(e^\rho) \mathrel{|} 1\leq i \leq n\}$ of $A[\Lambda]^W$ both have computational advantages. The $S$-type elements have simpler expressions and they are used to define the generalized Chebyshev polynomials. On the other hand the $J$-type quotients are the characters of irreducible representations. We have the following celebrated theorem:

\begin{theorem}[Weyl character formula]\label{Weyl-Char-Form} Let $\chi_\lambda$ be the character of an irreducible representation of $\mf{g}$ with highest weight $\lambda$. Then $\chi_\lambda = J(e^{\rho+\lambda})/J(e^\rho)$. 
\end{theorem}
We finish this section by giving an example that illustrates the connection between the polynomial and the exponential invariants. 
\begin{example} Consider the $G_2$ case. We use the coordinate functions $a$ and $b$ attached to the coroots in the symmetric algebra $S(E^*)$. More precisely, we define $a$ and $b$ by $\gamma = a\alpha_1^\vee + b\alpha_2^\vee$. The Weyl group is generated by the transformations 
$$\sigma_{\alpha_1}:(a,b)\mapsto (-a+b,b) \quad \text{and} \quad \sigma_{\alpha_2}:(a,b)\mapsto (a,3a-b).$$
The following polynomials are invariant under the action of $W=\langle \sigma_{\alpha_1}, \sigma_{\alpha_2} \rangle$:
\begin{align*}
	f_1 &= 3a^2-3ab+b^2,\\
	f_2 &= 4a^6 - 12ba^5 + 13b^2a^4 - 6b^3a^3 + b^4a^2.
\end{align*}
Moreover, the polynomials $f_1$ and $f_2$ generate $\R[a,b]^W$ as a polynomial algebra. The degrees of $f_1$ and $f_2$ form the set $\{d_1,d_2\}=\{2,6\}$ as it is stated in Table~\ref{deg-pol-inv}.

Now let us consider the exponential analogue. The functions $S(e^{\omega_1})(\gamma)$ and $S(e^{\omega_2})(\gamma)$ obtained by $e^\lambda(\gamma) \mapsto e^{(\lambda, \gamma)}$ are also invariant under the action of the Weyl group (the term $-2\pi i$ is omitted for simplicity). Their series expansion come with homogeneous polynomials that can be written in terms of $f_1$ and $f_2$. We give the first few terms of these infinite series: 
\begin{align*}
	S(e^{\omega_1})(\gamma) & = (e^{\omega_1}+ e^{-\omega_1}+ e^{\omega_1- \omega_2}+ e^{-\omega_1+ \omega_2}+ e^{2\omega_1- \omega_2}+ e^{-2\omega_1+ \omega_2})(\gamma)\\ 
	&= 6+2f_1+\frac{f_1^2}{6}+ \frac{f_1^3}{180}+ \frac{f_2}{120}+\ldots, \\
	S(e^{\omega_2})(\gamma) &= (e^{\omega_2} +e^{-\omega_2}+ e^{3\omega_1- \omega_2} + e^{-3\omega_2+\omega_2}+ e^{3\omega_1- 2\omega_2}+ e^{-3\omega_1+2\omega_2})(\gamma)\\ &=6+6f_1+\frac{3f_1^2}{2}+\frac{11f_1^3}{60}-\frac{9f_2}{40} + \ldots.
\end{align*}
Similar series expansions can be written for $\chi_{\omega_1} = S(e^{\omega_1}) +1 $ and $\chi_{\omega_2} = S(e^{\omega_1})+ S(e^{\omega_2}) + 2$.
\end{example}

\section{Generalized Chebyshev Polynomials}
The main reference for this section is \cite{hoffwith}. The formal exponential sums are considered as complex valued functions by putting
$$e^\lambda(\gamma) \mapsto e^{-2\pi i(\lambda, \gamma)}$$
as in \cite[Lemma~4.1]{hoffwith}. Moreover, the generalized cosine function is defined as 
\[y_i(\gamma)=S(e^{\omega_i})(\gamma) = \sum_{\mu \in W(\omega_i)}e^{-2\pi i(\mu, \gamma)}\]
Set $\mathbf{y}=(y_1,\ldots,y_n)$. The exponential form of Chevalley's theorem has the following consequence.
\begin{theorem}[\cite{veselov},\cite{hoffwith}]\label{veselov}
	With each semisimple complex Lie algebra $\mf{g}$ of rank $n$, there is an associated infinite sequence of polynomial mappings $P_\mf{g}^k, ~k \in {\N}$ determined from the conditions
	\[ \mathbf{y}(k\gamma)=P_\mf{g}^k(\mathbf{y}(\gamma)). \]
	All coefficients of the polynomials defining $P_\mf{g}^k$ are integers.
\end{theorem}
\begin{proof}
	Apply Theorem~\ref{expformche}, with $A=\mathbf{Z}$,  and $x_i  =S(e^{\omega_i})$.
\end{proof}

Recall that the Chebyshev polynomials $T_k$ act on cosine values in a manner that is very similar to the statement of the above theorem. For compatibility with the Lie algebra constructions, we consider a normalized version of Chebyshev polynomials that satisfy the following functional equation 
$$T_k(2\cos(\theta)) = 2\cos(k\theta).$$
We call the polynomials $P_{\mf{g}}^k$ of Theorem~\ref{veselov} as generalized Chebyshev polynomials because they coincide with $T_k$ if $\mf{g}$ is the unique simple Lie algebra of rank $n=1$. More precisely, we have
\[P_{A_1}^k = T_k,\quad  k\in\N.\]
We note that $W(\omega_1)=\{\omega_1,-\omega_1\}$ and $y_1(\gamma)=2\cos(2\pi u_1)$ where $\gamma = u_1\alpha_1^\vee$. The Chebyshev polynomials can be computed by the recurrence relation 
\[T_k(x) = xT_{k-1}(x)-T_{k-2}(x)\]
for $k \geq 2$. There are similar recurrence relations for the generalized Chebyshev polynomials \cite{withers}. There is an explicit formula for the coefficients of $T_k$:
\[ T_k(x) = \sum_{j=0}^{\lfloor k/2\rfloor} \frac{k}{k-j}\binom{k-j}{j}
(-1)^j x^{k-2j}.\]
This formula can be proved by using Waring's formula \cite[Equation~(7.5)]{lidlnied}. Moreover, this idea can be generalized to write the coefficients of $P_{\mf{g}}^k$. The computations for the rank two cases, namely $A_2$, $B_2$, and $G_2$, are done in \cite{aydogdu}.

\begin{example}\label{exampleche} 
Consider the $G_2$ case. We set $y_1= S(e^{\omega_1})$ and $y_2 = S(e^{\omega_2})$. The exponential invariants $S(e^{2\omega_1})$ and $S(e^{2\omega_2})$ can be written as polynomials in $y_1$ and $y_2$ by Theorem~\ref{veselov}. A lengthy but straightforward computation gives that
	\begin{align*}
		S(e^{2\omega_1}) &= g_1(y_1,y_2) =y_1^2-2y_2-2y_1-6,\\
		S(e^{2\omega_2}) &= g_2(y_1,y_2) = y_2^2-2y_1^3+6y_1y_2+10y_2+18y_1+18.
	\end{align*}
The Jacobian matrix of $P_{G_2}^2=(g_1,g_2)$ is given by
\[ \frac{\partial (g_1,g_2)}{\partial (y_1,y_2)} = \left[\begin{array}{cc}2y_1-2&-2\\-6y_1^2+6y_2+18& 6y_1+2y_2+10\end{array}\right] \]	
Our main result, namely Theorem~\ref{main}, gives a practical method to write this matrix in the following form
\[ J(P_{G_2}^2) = \left[\begin{array}{cc} 2\chi_{\omega_1} -4\chi_0 &-2\chi_0\\-6\chi_{2\omega_1}& 4\chi_{\omega_1}+2\chi_{\omega_2}+2\chi_0\end{array}\right] \]
without finding the polynomials $g_i$. Recall that $\chi_\lambda$ is the character of an irreducible representation with highest weight $\lambda$. It can be computed with the help of the Weyl character formula. For instance 
\[\chi_0=1,\quad \chi_{\omega_1}=y_1+1,\quad \chi_{\omega_2}=y_1+y_2+2,\quad \chi_{2\omega_1}=y_1^2-y_2-3.\]
The determinant of the above matrix turns out to be 
\[4\chi_{\omega_1+\omega_2} = 4(y_1y_2+2y_1+2y_2+4).\] 
This is a special case of Theorem~\ref{twokinds}.
\end{example}

\section{Main Results}
Let $P_\mf{g}^k=(g_1,\ldots, g_n)$ be the generalized Chebyshev polynomials defined in the previous section. Our purpose is to understand the Jacobian matrix
\[J(P_\mf{g}^k)=\frac{\partial (g_1,\ldots,g_n)}{\partial (y_1,\ldots,y_n)}= \begin{bmatrix}
	\dfrac{\partial g_1}{\partial y_1} & \cdots & \dfrac{\partial g_1}{\partial y_n}\\
	\vdots                             & \ddots & \vdots\\
	\dfrac{\partial g_n}{\partial y_1} & \cdots & \dfrac{\partial g_n}{\partial y_n}
\end{bmatrix} \]
Recall that the derivation $D_j$ can be regarded as a multiple of $\partial u_j$ by using the coordinates $u_j = (\gamma,\alpha_j^\vee)$ and putting $e^\lambda(\gamma) \mapsto e^{-2\pi i(\lambda, \gamma)}$. Applying the chain rule, we get
\[ \frac{\partial(y_1(k\gamma),\ldots,y_n(k\gamma))}{\partial(u_1,\ldots,u_n)} = \frac{\partial(y_1(k\gamma),\ldots,y_n(k\gamma))}{\partial(y_1(\gamma),\ldots,y_n(\gamma))}  \frac{y_1(\gamma),\ldots,y_n(\gamma))}{\partial(u_1,\ldots,u_n)}. \] 
The following matrix appears frequently in our computations.
\begin{definition} We define $\Jac(k)=D_j(S(e^{k\omega_i}))$ for each $k\in \N$. 
\end{definition}

According to the chain rule above, we have $\Jac(k) = J(P_\mf{g}^k) \Jac(1)$. It follows that the Jacobian matrix for $P_\mf{g}^k=(g_1,\ldots,g_n)$ is given by
\[J(P_\mf{g}^k) =\Jac(k)\Jac(1)^{-1}.\]
The Weyl character formula has the following consequence.

\begin{theorem}\label{twokinds}
$\det(J(P_\mf{g}^k)) = k^n\chi_{(k-1)\rho}$
\end{theorem}
\begin{proof} Recall that $\det(\Jac(1))=J(e^\rho)$ by Theorem~\ref{steinberg}. Applying the chain rule, we see that $\det(\Jac(k)) = k^nJ(e^{k\rho})$. Finally, we have $J(e^{k\rho})/J(e^\rho) = \chi_{(k-1)\rho}$ by Theorem~\ref{Weyl-Char-Form}. 
\end{proof}

\begin{example} The Chebyshev polynomials of the first kind, denoted $T_k$ and of the second kind, denoted by $U_k$, can be defined by the following equations:
\[ T_k(2\cos(\theta)) = 2\cos(k\theta), \quad \text{and} \quad U_{k-1}(2\cos(\theta))=2\sin(k\theta)/2\sin(\theta).  \] 
Recall that we have the equality $P_{A_1}^k=T_k$. The Chebyshev polynomials of different kinds are related to each other by the identity
\[\frac{d}{dx}(T_k(x))=k  U_{k-1}(x)\]
Thus Theorem~\ref{twokinds} can be thought as a generalization of this identity, to the higher ranks. 
\end{example}

The adjugate matrix is the transpose of the cofactor matrix by definition. The product of a matrix with its adjugate gives a diagonal matrix whose diagonal entries are the determinant of the original matrix. For instance, we have
\[\Jac(1)\Adj(\Jac(1))=J(e^\rho)I\]
where $I$ is the $n\times n$ identity matrix. The computation of the cofactor matrix, and therefore the adjugate matrix, is difficult in general. In our context, the adjugate matrix $\Adj(\Jac(1))$ is rather easy to find. For this purpose, we start with writing $\Jac(1)$ explicity. 

\begin{theorem}\label{jac(1)}
	The Jacobian matrix $\Jac(1)=D_j(S(e^{\omega_i}))$ is given by
	\[\Jac(1)=\left[ \frac{1}{s_i}\sum_{w\in W} (w(\omega_i),\alpha_j^\vee)e^{w(\omega_i)} \right]\]
	where $s_i$ is the size of the stabilizer group $\Stab(\omega_i)$. 
\end{theorem}
\begin{proof}
The result follows immediately, once the derivation $D_j$ is applied to 
\[S(e^{\omega_i}) = \sum_{\mu \in W(\omega_i)}e^\mu = \frac{1}{s_i}\sum_{w\in W} e^{w(\omega_i)}.\] 
\end{proof}

Our next step is the computation of the adjugate matrix $\Adj(\Jac(1))$.
In the proof, we will need the following lemma.

\begin{lemma}\label{lambdaisrho} 
Let $\lambda =w_1(\omega_i)+w_2(\rho-\omega_j)$ for some $w_1, w_2 \in W$. The element $\lambda$ is in the fundamental Weyl chamber $\mf{C}$ if and only if  $\lambda = \rho$. This is possible if only if $\omega_i = \omega_j$, $w_1\in \Stab(\omega_i)$, and $w_2\in \Stab(\rho-\omega_i)$. 
\end{lemma}
\begin{proof}
	Recall that $\rho=\omega_1+\ldots+\omega_n$. The element $\lambda=w_1(\omega_i)+w_2(\rho-\omega_j)$ is integral by definition. If $\lambda\in \mf{C}$, then we have $\rho \leq \lambda$. On the other hand, we have $w(\mu) \leq \mu$ for $\mu \in \overline{\mf{C}}$, $w\in W$. Thus
	\[\rho \leq \lambda =w_1(\omega_i)+w_2(\rho-\omega_j) \leq \omega_i+\rho-\omega_j.\] 
	Using the inequalities $\rho \leq \lambda \leq \omega_i+\rho-\omega_j$ and the fact that $\lambda\in \mf{C}$, we conclude that $\omega_i=\omega_j$ and $\lambda =\rho$. 
	
	If $\rho=w_1(\omega_i)+w_2(\rho-\omega_i)$, then $w_1(\omega_i) \leq \omega_i$ implies that $w_2(\rho - \omega_i) \geq \rho - \omega_i$. This is possible only if $w_2$ leaves $\rho-\omega_i$ unchanged. It follows that $w_2\in \Stab(\rho-\omega_i)$. Moreover, we must have $w_1(\omega_i)=\omega_i$. This means that $w_1\in \Stab(\omega_i)$.
\end{proof}

\begin{theorem}\label{adjugate}
	The adjugate matrix of $\Jac(1)$ is given by
	\[\Adj(\Jac(1))=\left[ \frac{1}{2}\sum_{w\in W} \det(w)  (\omega_i,w(\alpha_j^\vee))e^{w(\rho-\omega_j)} \right]\]
\end{theorem}
\begin{proof}
Let $A=\Jac(1)$ and let $B$ be the matrix in the hypothesis. We rewrite the indices of their entries in a suitable way for the matrix multiplication
\begin{align*}
	A_{im} &= \frac{1}{s_i}\sum_{w_1\in W} (w_1(\omega_i),\alpha_m^\vee)e^{w_1(\omega_i)}, \text{and}\\
	B_{mj}&= \frac{1}{2} \sum_{w_2\in W} \det(w_2)  (\omega_m,w_2(\alpha_j^\vee)) e^{w_2(\rho-\omega_j)}.
\end{align*}
The entries of the product $AB=C$ are given by $C_{ij} = \sum_{m=1}^n A_{im}B_{mj}$. By Lemma~\ref{inner}, we have
\[ \sum_{m=1}^{n}(w_1(\omega_i),\alpha_m^\vee)(\omega_m, w_2(\alpha_j^\vee)) = (w_1(\omega_i),w_2(\alpha_j^\vee)).\]
Therefore, the entries of the matrix $C$ are given by 
\[C_{ij}= \frac{1}{2s_i}\sum_{w_1\in W} \sum_{w_2\in W} \det(w_2)(w_1(\omega_i) ,w_2(\alpha_j^\vee)) e^{w_1(\omega_i)+w_2(\rho - \omega_j)}.\]
We claim that each $C_{ij}$ is anti-invariant. To see this, we set $w_1'=ww_1$ and $w_2'=ww_2$ for $w\in W$. We have 
\[w(C_{ij})= \frac{1}{2s_i}\sum_{w_1'\in W} \sum_{w_2'\in W} \det(w_2)(w_1(\omega_i) ,w_2(\alpha_j^\vee)) e^{w_1'(\omega_i)+w_2'(\rho - \omega_j)} \]
We observe two things. Firstly, the element $w\in W$ is an isometry, and we have 
$$(w_1(\omega_i) ,w_2(\alpha_j^\vee)) = (ww_1(\omega_i) ,ww_2(\alpha_j^\vee))= (w_1'(\omega_i) ,w_2'(\alpha_j^\vee))$$
Secondly, $\det(w)$ is either one or minus one. Thus 
$$\det(w_2) = \det(w_2')/\det(w) = \det(w_2')\det(w).$$ 
It follows that $w(C_{ij})=\det(w)C_{ij}$. This finishes the proof of the fact that $C_{ij}$ is anti-invariant.

The elements $J(e^\lambda)$ with $\lambda \in \Lambda \cap \mf{C}$ form a basis for the submodule of anti-invariant elements. The element $\lambda=w_1(\omega_i)+w_2(\rho - \omega_j)$ is in the fundamental Weyl chamber $\mf{C}$ if and only if  $\lambda = \rho$ and $\omega_i=\omega_j$ by Lemma~\ref{lambdaisrho}. It follows immediately that $C_{ij}=0$ if $i\neq j$. 

The significance of the division by $s_i$ and $2$ comes from the fact that $\lambda = \rho$ if only if $\omega_i = \omega_j$, $w_1\in \Stab(\omega_i)$, and $w_2\in \Stab(\rho-\omega_i)$. Recall that $s_i = |\Stab(\omega_i)|$ and $2=|\Stab(\rho-\omega_i)|$. It follows that $C_{ii}$ is an anti-invariant element with unique maximal term $e^\rho$. We must have $C_{ii}=J(e^\rho)$. 
\end{proof}

We want to express our main result for the entries $\partial g_i/\partial y_j$ of the matrix $J(P_\mf{g}^k)$ in a simple form. For this purpose, we make the following definition.
\begin{definition}\label{dij}
For each pair $(w_1,w_2)\in W^2$, we define
$$d_{ij}^k(w_1,w_2) = \left\{\begin{array}{cl}
	\det(w_2)(w_1(k\omega_i) ,w_2(\alpha_j^\vee)) & \textnormal{if } w_1(k\omega_i)+w_2(\rho - \omega_j) \in \Lambda \cap \mf{C},\\
		0& \textnormal{otherwise}.
\end{array}\right.$$
\end{definition}
If $k=1$, the integer $d_{ij}^k(w_1,w_2)$ is rather well understood by Lemma~\ref{lambdaisrho}. We have $d_{ij}^1(w_1,w_2)=1$ if and only if $\omega_i=\omega_j$, $w_1\in\Stab(\omega_i)$, and $w_2 \in \Stab(\rho-\omega_i)$.  Moreover, $d_{ij}^1(w_1,w_2)=0$, otherwise. The situation changes dramatically if $k>1$, but the computation of $d_{ij}^k$ can be done practically with the help of a computer by using the matrices $T_w$ of Lemma~\ref{glz}.

We now state our main result.
\begin{theorem}\label{main}
	The entries of the Jacobian matrix $J(P_\mf{g}^k)$ are given by
	$$\frac{\partial f_i}{\partial y_j} = \sum_{w_1\in W} \sum_{w_2\in W} \frac{d_{ij}^k(w_1,w_2)}{2s_i}\chi_{w_1(k\omega_i)+w_2(\rho-\omega_j)-\rho}.$$
\end{theorem}
\begin{proof}
Recall that $J(P_\mf{g}^k) =\Jac(k) \Jac(1)^{-1}$. Moreover $\Jac(1)\Adj(\Jac(1)) = J(e^\rho)$. It follows that 
\[J(P_\mf{g}^k) = \frac{1}{J(e^\rho)} \Jac(k) \Adj(\Jac(1)).\]
All entries of the matrix $J(P_\mf{g}^k)$ are invariant under $W$ since they can be expressed in terms of the invariant elements $y_1, \ldots, y_n$. On the other hand, each entry of the matrix $\Jac(k) \Adj(\Jac(1))$ is anti-invariant since it is obtained by multiplying the matrix $J(P_\mf{g}^k)$ with the anti-invariant element $J(e^\rho)$.

We want to understand the entries of the matrix $\Jac(k) \Adj(\Jac(1))$ in terms of the anti-invariant basis elements $J(e^\lambda)$, with $\lambda \in \Lambda \cap \mf{C}$. We have an explicit description for $\Jac(1)$ by Theorem~\ref{jac(1)}. Moreover, the chain rule implies that 
	\[\Jac(k)=\left[ \frac{1}{s_i}\sum_{w\in W} (w(k\omega_i),\alpha_j^\vee)e^{w(k\omega_i)} \right].\]
On the other hand, by Theorem~\ref{adjugate}, we have
	\[\Adj(\Jac(1))=\left[ \frac{1}{2}\sum_{w\in W} \det(w)  (\omega_i,w(\alpha_j^\vee))e^{w(\rho-\omega_j)} \right]\]
The rest of the proof follows closely the pattern of the proof of Theorem~\ref{adjugate}.	 Applying Lemma~\ref{inner} to the product of these two matrices, we see that the $ij$-th entry of the product is given by
\[\frac{1}{2s_i}\sum_{w_1\in W} \sum_{w_2\in W} \det(w_2)(w_1(k\omega_i) ,w_2(\alpha_j^\vee)) e^{w_1(k\omega_i)+w_2(\rho - \omega_j)}.\]
Each entry is anti-invariant and can be expressed as a linear combination of $J$-type basis elements $J(e^\lambda)$, with $\lambda \in \Lambda \cap \mf{C}$. The coefficients are captured by using the function $d_{ij}^k$. We finally divide everything by $J(e^\rho)$, a common divisor of anti-invariant elements. Using the Weyl character formula, i.e. Theorem~\ref{Weyl-Char-Form}, we get the desired expression. 
\end{proof}

The appearance of the determinant in Definition~\ref{dij} is essential. It occurs as a balancing factor between the action of the Weyl group on weights and coroots. If 
$$(\tau_1,\tau_2)\in w_1\Stab(\omega_i) \times w_2\Stab(\rho - \omega_j),$$ then $w_1(k\omega_i)+w_2(\rho - \omega_j) = \tau_1(k\omega_i)+\tau_2(\rho - \omega_j)$. Recall that the $\Stab(\rho - \omega_j)=\{1,\sigma_{\alpha_j}\}$. 
If we pick $\tau_2=w_2 \sigma_{\alpha_j}$, then $ \tau_2(\alpha_j^\vee) =  -w_2(\alpha_j^\vee)$. On the other hand $\det(\tau_2)= - \det(w_2)$. The minus signs cancel each other and there are $2s_i$ pairs $(w_1,w_2)$ in $W^2$ which give the same quantity. We illustrate this situation with the following example.

\begin{example} Consider the $G_2$ case with $i=j=1$ and $k=2$. Note that $\rho=\omega_1+\omega_2$ and $\rho-\omega_1 = \omega_2$. Recall that $\omega_1 = 2\alpha_1 + \alpha_2$ and $\omega_2 = 3\alpha_1 + 2\alpha_2$. We look for elements of the form $\lambda = w_1(2\omega_1) + w_2(\omega_2)$ in $\Lambda \cap \mf{C}$ between $\rho=5\alpha_1+3\alpha_2$ and $2\omega_1+ \omega_2= 7\alpha_1+4\alpha_2$. It turns out that there are precisely two.
	
We start with $\lambda_1 = 2\omega_1 + \omega_2$. In this case $(w_1,w_2)	\in \{1,\sigma_{\alpha_2}\}\times \{1,\sigma_{\alpha_1}\}$. We have $\lambda_1 = w_1(2\omega_1)+w_2(\omega_2)$, trivially. We have 
\[d_{11}^2(w_1,w_2)=\det(w_2)(2\omega_1 ,\pm\alpha_1^\vee)=2.\]

Secondly, $\lambda_2 = \omega_1 + \omega_2$. In this case, $(w_1,w_2)	\in \{\sigma_{\alpha_1},\sigma_{\alpha_1}\sigma_{\alpha_2}\}\times \{\sigma_{\alpha_2},\sigma_{\alpha_2}\sigma_{\alpha_1}\}$. For these pairs $(w_1,w_2)$, we note that $w_1(2\omega_1)=-2\omega_1+2\omega_2$ and $w_2(\omega_2)=3\omega_1-\omega_2$ whose sum is equal to $\lambda_2$. We have
\[d_{11}^2(w_1,w_2)=\det(w_2)(-2\omega_1+2\omega_2 ,\mp(\alpha_1^\vee +3\alpha_2^\vee))=-4.\]	
	
Our main result, namely Theorem~\ref{main}, implies that $$\partial g_1/\partial y_1 = 2\chi_{\omega_1}-4\chi_0$$ as expected from Example~\ref{exampleche}. This result is obtained without computing the polynomials $g_1$ and $g_2$.
\end{example}

We finish this manuscript by giving some examples of low rank for arbitrary $k$. For this purpose, we need to use the notion of highest root.

Recall that the matrices $T_w$ of Lemma~\ref{glz} has integer integers. The matrices $T_w$ has a symmetric nature when they act on the coordinate vectors of weights and coroots due to the following equation
\[T_w=[(\omega_i,w(\alpha_j^\vee))]=[(w^{-1}(\omega_i),\alpha_j^\vee)]\]
If we fix the basis $\mathcal{A}=\{\omega_1,\ldots,\omega_n\}$, then $W$ acts on the row vector $[\lambda]_\mathcal{A}$ by the right multiplication $[\lambda]_\mathcal{A} \cdot T_w$. On the other hand, if we fix the basis $\mathcal{B}=\{\alpha_1^\vee,\ldots,\alpha_n^\vee\}$, then $W$ acts on the column vector $[\gamma]_\mathcal{B}$ by the left multiplication $T_w^{-1}[\gamma]_\mathcal{B}$. 

\begin{table}[!h]
	\[\begin{array}{|cccccccccc|} \hline
		\textnormal{Type} & A_n & B_n & C_n & D_n & E_6 & E_7 & E_8 & F_4 & G_2\\ \hline
		m_\mf{g} & 1 & 2 & 2 & 2 & 3 & 4 & 6 & 4 & 3 \\ \hline
	\end{array}\]
	\caption{\label{highigh}The highest coefficient of the highest root.}
\end{table}

Let $m_\mf{g}$ be the maximum of the absolute values of the entries of $T_w$. The integer $m_\mf{g}$ turns out to be  the highest coefficient of the highest root. These integers can be found in \cite{springer} and listed in Table~\ref{highigh}.

If $k \geq m_\mf{g}$, then we claim that the term $w_1(k\omega_i)$ in Theorem~\ref{main} must be of the form $k\omega_i$ with $w_1\in\Stab(\omega_i)$. To see this, we first note that the row vector $[\rho- \omega_j]_\mathcal{A}$ has entries all one except a single zero. The largest coordinate we can get from the computation of $[\rho- \omega_j]_\mathcal{A} \cdot T_{w_2}$ is the integer $m_\mf{g}$. If $k \geq m_\mf{g}$, the negative coordinates brought by $w_1(k \omega_i)$ cannot be cancelled by $w_2(\rho-\omega_j)$, and the resulting $\lambda=w_1(k\omega_i)+w_2(\rho-\omega_j)$ cannot be in the chamber $\mf{C}$. Observe that the below formula for $G_2$ is true for $k\geq 3$ but not for $k=2$. 

We use the shorter expressions $\chi_{a}$, $\chi_{a,b}$, and $\chi_{a,b,c}$ instead of $\chi_{a\omega_1}$, $\chi_{a\omega_1+b\omega_2}$, and $\chi_{a\omega_1+b\omega_2+c\omega_3}$ to save some space. The expressions with negative values of $a$,$b$, or $c$ are simply zero.
\[J(P_{A_1}^k) = k\chi_{k-1}\quad k\geq 1.\]

\[J(P_{A_2}^k) = k \begin{bmatrix}
	\chi_{k-1,0} & -\chi_{k-2,0} \\
	-\chi_{0,k-2} & \chi_{0,k-1}
\end{bmatrix} \quad k\geq 1.\]

\[J(P_{B_2}^k) = k\begin{bmatrix}
	\chi_{k-1,0}+\chi_{k-3,0} & -\chi_{k-2,0} \\
	-2\chi_{1,k-2} & \chi_{0,k-1}+\chi_{0,k-2}
\end{bmatrix} \quad k\geq 2.\]

\[J(P_{G_2}^k) = k\begin{bmatrix}
	\chi_{k-1,0}+\chi_{k-4,0}+2 \chi_{k-4,1} & -\chi_{k-2,0}-\chi_{k-3,0} \\
	-3\chi_{2,k-2}-3\chi_{2,k-3} & \chi_{0,k-1}+\chi_{0,k-2}+2 \chi_{1,k-2}
\end{bmatrix} \quad k\geq 3.\]

\[J(P_{A_3}^k) = k\begin{bmatrix}
	\chi_{k-1,0,0}&-\chi_{k-2,0,0}&\chi_{k-3,0,0} \\
	\chi_{1,k-3,0}-\chi_{0,k-2,1}&\chi_{0,k-1,0}-\chi_{0,k-3,0}&\chi_{0,k-3,1}-\chi_{1,k-2,0} \\
	\chi_{0,0,k-3}&-\chi_{0,0,k-2}&\chi_{0,0,k-1}
\end{bmatrix} \quad k\geq 1.\]

\section{Acknowledgement}
The second author joined the University of Massachusetts, Amherst Mathematics department as a graduate student in 2004, shortly after James E. Humphreys became a professor of emeritus, and had a few short talks with him during the colloquiums. Unfortunately, James E. Humphreys passed away in 2020 because of the pandemic. Both authors wish condolences for his loss and acknowledge his contribution to the theory of Lie algebras. Without his excellent books, the preparation of this manuscript might not be possible.

{\small
\def\refname{References}
\newcommand{\etalchar}[1]{$^{#1}$}


\begin{thebibliography}{tt}
	
\bibitem[Ay21]{aydogdu} M. Aydo\u{g}du, \textit{Coefficients of folding polynomials attached to Lie algebras of rank two}, Master Thesis, Middle East Technical University (2021).

\bibitem[Bo72]{Bourbaki} N. Bourbaki, \textit{Elements de Math\`ematique, Groupes et Algebres de Lie}, Hermann, Paris, (1972).

\bibitem[Ca72]{carter}R. W. Carter, \textit{Simple groups of Lie type.} Wiley Classics Library. New York (1989).

\bibitem[Ch55]{chevalley}
Chevalley, Claude, \textit{Invariants of finite groups generated by reflections.} Amer. J. Math. 77 (1955), 778--782.

\bibitem[Fri70]{fried} M. Fried, \textit{On a conjecture of Schur.} Michigan Math. J. (1970), 17, 41--55.

\bibitem[GMS03]{guralnick}
R.~M.~Guralnick, P.~Müller; J.~Saxl, \textit{The rational function analogue of a
 question of Schur and exceptionality of permutation representations.} (English
summary) Mem. Amer. Math. Soc. 162 (2003), no. 773. 

\bibitem[HW88]{hoffwith}
M. E. Hoffman and W. D. Withers; \textit{Generalized Chebyshev polynomials 
associated with affine Weyl groups.} Trans. Amer. Math. Soc. 308 (1988), 
91--104.

\bibitem[Hu78]{hump-lie} Humphreys, James E. \textit{Introduction to Lie algebras and representation theory.} Second printing, revised. Graduate Texts in Mathematics, 9. Springer-Verlag, New York-Berlin, 1978.

\bibitem[Hu90]{hump-ref} Humphreys, James E.
\textit{Reflection groups and Coxeter groups.} Cambridge Studies in Advanced Mathematics, 29. Cambridge University Press, Cambridge, 1990.

\bibitem[K\"{u}18]{kucuk-bulletin} 
\"{O}. K\"{u}\c{c}\"{u}ksakall\i, \textit{On the arithmetic exceptionality of polynomial mappings.} Bull. Lond. Math. Soc. 50 (2018), no. 1, 143--147.

\bibitem[LN83]{lidlnied}
R.~Lidl, H.~Niederreiter, \textit{Finite fields, Encyclopedia of Mathematics and
	its Applications, Vol. 20.} Cambridge, UK: Cambridge University Press, (1983).

\bibitem[LW72]{lidlwells}
R.~Lidl, C.~Wells, \textit{Chebyshev polynomials in several variables.}
J. Reine Angew. Math. 255 (1972), 104--111.

\bibitem[Sp66]{springer} T. A. Springer, \textit{Some arithmetical results on semi-simple Lie algebras.} Inst. Hautes \'Etudes Sci. Publ. Math. No. 30 (1966), 115--141.

\bibitem[Ve87]{veselov}
A. P. Veselov, \textit{Integrable mappings and Lie algebras.} Soviet Math. 
Dokl. 35 (1987), 211--213.

\bibitem[Ve91]{veselov-survey}
A. P. Veselov, \textit{Integrable mappings.} Russian Math. Surveys 46 (1991), no. 5, 1--51.

\bibitem[Wi88]{withers}W. D. Withers, \textit{Folding polynomials and their dynamics.} Amer. Math. Monthly 95 (1988), no. 5, 399--413.

\end{thebibliography}
\end{document}